\documentclass{amsart}
\usepackage[margin=1in]{geometry}

\usepackage{tikz}
\usepackage{mathdots}
\usepackage{array}
\usepackage{multirow}
\usepackage{hyperref}
\usepackage{float}
\usepackage{caption}
\usepackage{amscd}
\usepackage{tabularx}
\usepackage{extarrows}
\usepackage{ytableau}
\usepackage{wasysym}

\usepackage{amsmath}
\usepackage{amsthm}
\usepackage{amssymb}
\usepackage{mathtools}
\usepackage{soul}
\usepackage{color}

\newtheorem{theorem}{Theorem}[section]
\newtheorem{conjecture}[theorem]{Conjecture}
\newtheorem{proposition}[theorem]{Proposition}
\newtheorem{corollary}[theorem]{Corollary}
\newtheorem{lemma}[theorem]{Lemma}
\theoremstyle{definition}
\newtheorem{definition}[theorem]{Definition}
\newtheorem{example}[theorem]{Example}

\newcommand{\N}{\mathbb{N}}
\newcommand{\Z}{\mathbb{Z}}

\DeclareMathOperator{\rr}{\mathfrak{r}}

\title{Counting and Constructing Latin tableaux}

\author[S. Daugherty]{Spencer Daugherty}\address{Department of Mathematics, University of Colorado Boulder, Boulder, CO}\email{spencer.daugherty@colorado.edu}

\author[M. Fahnestock]{Michaela Fahnestock}\address{Department of Mathematics, University of Colorado Boulder, Boulder, CO}\email{michaela.fahnestock@colorado.edu}

\author[S. Huxtable]{Shawn Huxtable}\address{Department of Mathematics, University of Colorado Boulder, Boulder, CO}\email{amhu7380@colorado.edu}

\author[M. Larson]{Michael Larson}\address{Department of Mathematics, University of Colorado Boulder, Boulder, CO}\email{michael.larson@colorado.edu}

\author[N. Al Munayes]{Nourah Al Munayes}\address{Department of Mathematics, University of Colorado Boulder, Boulder, CO}\email{nourah.almunayesjr@colorado.edu}

\author[S. Peacock]{Storey Peacock}\address{Department of Mathematics, University of Colorado Boulder, Boulder, CO}\email{storey.peacock@colorado.edu}

\usepackage[backend=bibtex, maxbibnames=10, doi=false, url=false, isbn=false]{biblatex}
\begin{document}

 \begin{abstract} Latin tableaux generalize Latin squares to Young diagrams and have interesting connections to linear algebra and Rota's basis conjecture. Conjecturally, a partition admits Latin tableaux if and only if it belongs to a special class called wide partitions. 
We prove that various classes of wide partitions admit Latin tableaux. We extend these classes further by showing that the number of Latin tableaux is weakly increasing over certain operations on wide partitions. 
Additionally, we establish an explicit formula for the number of Latin tableaux using (0,1)-matrices and matrix permanents. 
 \end{abstract}

\maketitle

\section{Introduction}

An $n \times n$ \emph{Latin square} is an $n \times n$ grid filled with $n$ different symbols such that each symbol appears exactly once in each row and once in each column. Latin squares were first explored by Choi Seok-jeong \cite{ree2014choi} and famously studied by Leonhard Euler \cite{Euler1782}, both in the 1700s. Unless otherwise specified, we consider the $n$ symbols in a Latin square to be $\{1,2,\ldots, n\}$, a set we denote by $[n]$ throughout. 

\begin{example}\label{ex:LS} The 12 distinct Latin squares of size 3 are: \vspace{2mm}
    \[\begin{ytableau}
    1 & 2 & 3 \\
    2 & 3 & 1 \\
    3 & 1 & 2
       \end{ytableau} \qquad 
       \begin{ytableau}
           1 & 2 & 3 \\
           3 & 1 & 2\\
           2 & 3 & 1
       \end{ytableau} \qquad 
     \begin{ytableau}
    1 & 3 & 2 \\
    2 & 1 & 3 \\
    3 & 2 & 1
       \end{ytableau} \qquad 
       \begin{ytableau}
           1 & 3 & 2 \\
           3 & 2 & 1\\
           2 & 1 & 3
       \end{ytableau} \qquad \begin{ytableau}
    2 & 1 & 3 \\
    1 & 3 & 2 \\
    3 & 2 & 1
       \end{ytableau} \qquad 
       \begin{ytableau}
           2 & 1 & 3 \\
           3 & 2 & 1\\
           1 & 3 & 2
       \end{ytableau} 
       \] \vspace{2mm}
       \[
\begin{ytableau}
    2 & 3 & 1 \\
    1 & 2 & 3 \\
    3 & 1 & 2
       \end{ytableau} \qquad 
       \begin{ytableau}
           2 & 3 & 1 \\
           3 & 1 & 2\\
           1 & 2 & 3
       \end{ytableau} \qquad 
       \begin{ytableau}
    3 & 1 & 2 \\
    1 & 2 & 3 \\
    2 & 3 & 1
       \end{ytableau} \qquad 
       \begin{ytableau}
           3 & 1 & 2 \\
           2 & 3 & 1\\
           1 & 2 & 3
       \end{ytableau} \qquad
       \begin{ytableau}
    3 & 2 & 1 \\
    1 & 3 & 2 \\
    2 & 1 & 3
       \end{ytableau} \qquad 
       \begin{ytableau}
           3 & 2 & 1 \\
           2 & 1 & 3\\
           1 & 3 & 2
       \end{ytableau} 
       \] \vspace{-2mm}
\end{example}

Latin squares have been a major topic of study in discrete math over the last 300 years. This includes applications in experimental design, error-correcting codes, cryptography, and algebra, where they are the multiplication tables of quasigroups \cite{keedwell2015latin}. A major problem historically has been enumerating Latin squares - all existing formulas for counting Latin squares are computationally prohibitive, and values are only known up to $n=11$ (see OEIS \href{https://oeis.org/A002860}{A002860}). 

A natural way to generalize Latin squares is to extend the concept to all Young diagrams.  A \emph{Latin tableau} of shape $\lambda$ and type $\lambda'$ is a diagram of $\lambda$ filled with integers such that row $i$ contains exactly the entries $1$ through $\lambda_i$ and no entry repeats within a column. These were first introduced in \cite{chow2003wide} alongside the question: which partition shapes admit Latin tableaux? This question translates nicely to graph theory as a list-edge-coloring problem. For a partition $\lambda$, let $G(\lambda)$ be the bipartite graph whose vertices represent the rows and columns of $\lambda$ and whose edges represent the cells of $\lambda$. The Latin tableaux of shape $\lambda$ correspond exactly to the proper list-edge-colorings of $G(\lambda)$ where the list at cell edge $(i,j)$ is $\{1,2, \ldots, \lambda_i\}$ where the color on each edge corresponds to the symbol in the appropriate cell of $\lambda$. When $\lambda$ is rectangular, this reduces to the Dinitz problem, which is solved by Galvin \cite{GALVIN1995153}.

Chow-Fan-Goemans-Vondrak conjecture that the class of partitions which admit Latin tableaux (which they call Latin partitions) is exactly the class of partitions they refer to as \emph{wide}. A partition $\lambda$ is wide if every subpartition dominates its conjugate. This claim is known as the wide partition conjecture on free matroids.

\begin{conjecture}[WPC on free matroids]\label{conj:wpcfm}
A partition $\lambda$ is Latin if and only if it is wide.
\end{conjecture}

The WPC for free matroids is a special case of a more general conjecture, the wide partition conjecture, which covers all matroids. The WPC was developed by Timothy Chow and Brian Taylor as a generalization of a famous conjecture of Rota from 1989 \cite{huang1994relations}. This conjecture is still unsolved despite significant attention. We state the conjecture here both in terms of vector spaces and in terms of matroids.

\begin{conjecture}[Rota's Basis Conjecture] \label{conj:rota}
    Let $\mathbb{F}$ be a field and let $B_1, \ldots, B_n \subseteq \mathbb{F}^n$ be bases of the vector space $\mathbb{F}^n$. Then the multiset $B_1 \cup \cdots \cup B_n$ can be partitioned into $n$ bases of $\mathbb{F}^n$, which are transversal with respect to the original bases $B_1, \ldots, B_n$. Equivalently, given $n$ disjoint bases $B_1, \ldots, B_n$ in a rank $n$ matroid $M$, there exist $n$ disjoint transversals $A_1, \ldots, A_n$ of $B_1, \ldots, B_n$ that are all bases of $M$. 
\end{conjecture}

Conjecture \ref{conj:rota} is known to be true for paving matroids \cite{geelen2006rota}, strongly base orderable matroids \cite{wild1994rota}, random bases of vector spaces \cite{sauermann2022rotasbasisconjectureholds}, and asymptotically \cite{pokrovskiy2020rotasbasisconjectureholds}.
It is also true over fields of characteristic zero if $n-1$ or $n+1$ is a prime number. This follows from observations in \cite{huang1994relations} and the proofs of the Alon-Tarsi conjecture \cite{alon1992colorings} for these cases by Drisko \cite{drisko1997number} and Glynn \cite{glynn2010conjectures}.  This connects back to Latin squares, as the Alon-Tarsi conjecture claims that the number of even Latin squares of order $n$ is not equal to the number of odd Latin squares of order $n$, where parity is determined by the product of the signs of the permutations that make up its rows. 

The backward direction of the WPC for free matroids (Latin $\Rightarrow$ wide) is known to be true by an argument of Vic Reiner:  Let $T$ be a Latin tableau of shape $\lambda$ and type $\lambda'$, and note that for any $\mu \subseteq \lambda$, $T$ restricts to $\mu$ as a Latin tableau $S$ of shape $\mu$ and content $\mu'$. To show that $\mu \succeq \mu'$, fix $i$ and delete all entries of $S$ that are strictly greater than $i$ (leaving only the entries corresponding to $\mu_1', \ldots, \mu_i'$).  If we top-justify the remaining entries, they must all fit into the first $i$ rows of $S$. Thus, the $\mu_1 + \mu_2 + \cdots + \mu_i \geq \mu_1'+\mu_2' + \cdots + \mu_i'$.

The forward direction (wide $\Rightarrow$ Latin) is only proven for certain families of wide partitions, although it has been verified computationally for partitions of size $n \leq 65$. Chow-Fan-Goemans-Vondrak prove Latinness for partitions with at most two distinct part sizes, certain families with 3 or 4 distinct part sizes, and self-conjugate partitions with at most three distinct part sizes. They also reduce the conjecture to indecomposable wide partitions, and further to self-conjugate wide partitions. Allsop-Kotlar-Wanless prove that all wide partitions with exactly three distinct row lengths are Latin \cite{allsop2025outlinerectanglesallocationslatin}.

In this paper, we prove that various new classes of wide partitions are Latin. Our classes differ significantly in form from those studied in \cite{allsop2025outlinerectanglesallocationslatin, chow2003wide}, which both focus on cases with a small number of distinct parts. After providing a formula for the number of Latin tableaux, we also use enumerative methods to further extend the classes of wide partitions we know to be Latin. Specifically, we prove that certain changes to a partition always result in a weak increase in the number of Latin tableaux, thus allowing us to build on our base classes.
In Section \ref{sec:pre}, we cover background material and present a few structural results on wide partitions. In Section \ref{sec:WPC}, we prove that the following three distinct families of wide partitions are Latin:
\begin{enumerate}
    \item $\lambda= (\lambda_1, \ldots, \lambda_{\ell})$ with $\lambda_{\ell} \geq \ell$,
    \item $\lambda = (\lambda_1, \ldots, \lambda_{\ell})$ with $\ell-1$ distinct parts where $i$ is the part that appears twice, and there exists an $m \in \{m < i : m \not= \lambda_j,j\in[\ell]\}$ such that $m+i \not= \lambda_j$ for any $j \in [\ell]$, and 
    \item $\lambda = (\lambda_1, \ldots, \lambda_{\ell})$ where the largest multiplicity of any part is greater than or equal to $gcd(\lambda_1, \ldots, \lambda_{\ell})$.
\end{enumerate}

In Section \ref{sec:LT}, we give the first explicit formula for the number of Latin tableaux of shape $\lambda$ and type $\lambda'$.  When $\lambda$ is a rectangle or square, we recover classic formulas of Shao-Wei \cite{SHAO1992293}. In Section \ref{sec:mono}, we study operations on a partition $\mu$ that produce a partition $\lambda$ where $LT(\lambda) \geq LT(\mu)$. These extend the classes of wide partitions we know to be Latin. We establish that adding a box to the first row, adding a long-enough new first row, and deleting a single-box last row all weakly increase the number of Latin tableaux on the shape. 

\begin{table}[H]
    \centering
    \begin{tabular}{|c||c|c|c|c|c|c|c|c|} \hline
      & $n=21$ & $n=22$ & $n=23$ & $n=24$ & $n=25$ & $n=26$ & $n=27$ & $n=28$\\ \hline \hline
       \S \ref{sec:WPC}, \ref{sec:mono} & 139 & 167 & 194 & 239 & 278 & 335 & 389 & 468 \\ \hline
        \cite{allsop2025outlinerectanglesallocationslatin, chow2003wide}  & 146 & 176 & 207 & 251 & 295 & 352 & 415  & 490 \\ \hline
        \S \ref{sec:WPC}, \ref{sec:mono} + \cite{allsop2025outlinerectanglesallocationslatin, chow2003wide}  & 149 & 183 & 215 & 264 & 310 & 375 & 441  & 527\\ \hline
       \S \ref{sec:WPC}, \ref{sec:mono} $-$ \cite{allsop2025outlinerectanglesallocationslatin, chow2003wide}   & 3 & 7 & 8 & 13 & 15 & 23 & 26  & 37\\ \hline \hline
       Total Wide & 150 & 183 & 217 & 266 & 315 & 380 & 449  & 534\\ \hline 
    \end{tabular}
    \caption{Number of wide partitions proved to be Latin}
    \label{tab:vals}
\end{table}

Table \ref{tab:vals} contains, for $21\leq n \leq 28$, the number of wide partitions of size $n$ proved to be Latin by results in this paper and/or others. Note that we are referring to wide partitions as proved if they fall into a Latin family as given by a result, as all the partitions here are known to be Latin computationally. Row 1 counts the number of partitions proved to be Latin using results from Sections \ref{sec:WPC} and \ref{sec:mono} along with the fact that the component-wise sum of two Latin partitions is also Latin \cite[Prop 8]{chow2003wide}. Row 2 counts the number of partitions proved to be Latin using only results from \cite{allsop2025outlinerectanglesallocationslatin, chow2003wide}. Row 3 counts the number of partitions proved to be Latin using all results from both our paper and \cite{allsop2025outlinerectanglesallocationslatin, chow2003wide}. Row 4 counts the number of partitions proved to be Latin by results in Sections \ref{sec:WPC} and \ref{sec:mono} that are not known to be Latin by any results in \cite{allsop2025outlinerectanglesallocationslatin, chow2003wide}. The final row counts the number of wide partitions.
These
values are all computer-generated using Claude and SageMath. 

\begin{table}[H]
    \centering
    \begin{tabular}{|c|c|c|} \hline
    Result & First new shape & Size \\ \hline \hline
        Theorem \ref{thm:wide-strict-rep} & $(5,4,4,2,1)$ & 16 \\ \hline
         Corollary \ref{cor:addrow} & $(7,5,4,3,3)$ & 22 \\ \hline
         Theorem \ref{thm:gcd-mult} & $(8,8,6,4,4,2,2)$ & 34\\ \hline
         Theorem \ref{thm:fat} & $(9,9,8,7,6,6)$ & 45 \\ \hline
    \end{tabular}
    \caption{Smallest shapes proved Latin exclusively by given result}
    \label{tab:smallestval}
\end{table}

In the range considered in Table \ref{tab:vals}, every partition counted as new is obtained using Theorem \ref{thm:wide-strict-rep} and Corollary \ref{cor:addrow}. Theorem \ref{thm:fat} and Theorem \ref{thm:gcd-mult} start contributing at higher values of $n$. Table \ref{tab:smallestval} lists the smallest partition that is proved to be Latin exclusively by the listed result.

\subsection*{Acknowledgements} This paper began at the Summer REU/G at the University of Colorado at Boulder in 2025. The authors would like to thank Nat Thiem and the CU Boulder Department of Mathematics.

\section{Preliminaries}\label{sec:pre}

A \emph{partition} $\lambda = (\lambda_1, \ldots, \lambda_{k})$ of $n \in \mathbb{N}$ is a finite sequence of weakly decreasing positive integers $\lambda_1 \geq \lambda_2 \geq \cdots \geq \lambda_{k}$ such that $|\lambda| = \lambda_1 + \cdots + \lambda_{k} = n$.  If $|\lambda|=n$, we write $\lambda \vdash n$.  We let $\ell(\lambda)$ denote the number of parts in $\lambda$ so $\ell((\lambda_1, \ldots, \lambda_k)) = k$. The \emph{diagram} of a partition $\lambda$ is a left-aligned grid of boxes where the $i^{\text{th}}$ row from the top contains $\lambda_i$ boxes. The \emph{conjugate} of a partition $\lambda$ is defined $\lambda' = (\lambda'_1, \ldots, \lambda'_{\ell})$ where $\lambda'_i = |\{j \in [\ell] : \lambda_j \geq i\}|$. This corresponds to flipping the diagram of $\lambda$ over the diagonal that goes from the top left to the bottom right corner, thus swapping the rows and columns. The Durfee size of $\lambda$ is the number of boxes that lie on the aforementioned diagonal.

 We say that $\lambda$ \emph{dominates} $\mu$, denoted $\lambda \succeq \mu$, if $\lambda_1 + \cdots + \lambda_{i} \geq \mu_1 + \cdots \mu_{i}$ for all $i\in[\ell(\lambda)]$. We say $\mu$ is a \emph{subpartition} of $\lambda$, denoted $\mu \subseteq \lambda$, if $\mu$ can be obtained by deleting rows of $\lambda$. We say $\mu$ is a \emph{lower subparititon} of $\lambda$ if $\mu$ is obtained by deleting rows $1$ through $i$ from $\lambda$ for some $1 \leq i \leq \ell(\lambda)$.  Note that the term subpartition is also used to describe other concepts in the literature.

\begin{example} The partition $\lambda = (4,2,2)$ dominates the partition $\mu = (3,3,1)$ because $4 \geq 3$ and $ 4+2 \geq 3+3$ and $4 + 2 + 2 \geq 3 + 3 + 1$. The partition $(4,2,1)$ is a subpartition of the partition $(4,3,3,2,1)$, whereas the partition $(3,2,1)$ is a lower subpartition of $(4,3,3,2,1)$.
\end{example}

\begin{definition}
A partition $\lambda$ is \emph{wide} if $\mu \succeq \mu'$ for all $\mu \subseteq \lambda$. 
\end{definition}

The number of wide partitions is counted in OEIS \href{https://oeis.org/A070830}{A070830}, given by:
\[1, 1, 2, 3, 3, 5, 6, 9, 11, 14, 18, 23, 29, 35, 45, 56, 68, \ldots\]

\begin{example} The diagrams of the five wide partitions of size $6$ are
\[ \begin{ytableau}
    \ & \ & \ & \ & \ & \
\end{ytableau} \qquad \begin{ytableau}
    \ & \ & \ & \ & \ \\
    \
\end{ytableau} \qquad \begin{ytableau}
    \ & \ & \ & \ \\
    \ & \
\end{ytableau} \qquad \begin{ytableau}
    \ & \ & \ \\
    \ & \ & \
\end{ytableau} \qquad \begin{ytableau}
    \ & \ & \ \\
    \ & \ \\
    \
\end{ytableau}  \]
The partition $(4,1,1)$ of 6, on the other hand, is not wide because the subpartition $(1,1)$ does not dominate its conjugate, $(2)$. 
\end{example}

We now review some definitions and results on wide partitions.
For partitions $\mu$ and $\lambda$, let $ \lambda + \mu$ be the partition where $(\lambda+\mu)_i = \lambda_i + \mu_i$. A wide partition is \emph{decomposable} if it can be written as the sum of two wide partitions and \emph{indecomposable} otherwise.
The following statements are proved in \cite{chow2003wide}:
\begin{enumerate}
    \item If $\lambda$ is wide, then $\lambda_{\ell(\lambda)-i}>i$ for all $i \geq 0$.
    \item If $\lambda$ is wide and $\mu$ is a single column with height at most $\lambda_1'+1$, then $\lambda+\mu$ is wide.
    \item If $\lambda$ and $\mu$ are wide, then $\lambda+\mu$ is wide.
    \item   If $\lambda_{\ell} \geq \ell$, then $\lambda$ is wide.
\end{enumerate}

Notably, the proposition below offers a simpler condition for wideness. Moving forward, we will take the condition below as our criterion for a partition to be wide. From this definition, it is also clear that if $(\lambda_1, \lambda_2, \ldots, \lambda_k)$ is wide then $(\lambda_2, \lambda_3, \ldots, \lambda_k)$ is wide as well.

\begin{proposition}\cite{chow2003wide}
If $\lambda$ is a partition such that $\mu \succeq \mu'$ for all lower subpartitions $\mu$ of $\lambda$, then $\lambda$ is wide.
\end{proposition} 

\begin{definition} 
    A \emph{Latin tableau} of shape $\lambda$ and type $\mu$ is a diagram of $\lambda$ filled with integers such that $i$ appears exactly $\mu_i$ times and each integer appears at most once in any row or column.
\end{definition}

We are often interested in Latin tableaux of shape $\lambda$ and type $\lambda'$. In this case, each row $i$ contains exactly the entries $1$ through $\lambda_i$. A partition $\lambda$ is \emph{Latin} if there exists a Latin tableau of shape $\lambda$ and type $\lambda'$. 

\begin{example}
    The Latin tableaux of shape $(3,2)$ and type $(2,2,1)$ are:

    \[\begin{ytableau}
        1 & 2 & 3 \\
        2 & 1
    \end{ytableau}  \qquad
    \begin{ytableau}
        1 & 3 & 2 \\
        2 & 1
    \end{ytableau} \qquad
    \begin{ytableau}
        3 & 2 & 1 \\
        2 & 1
    \end{ytableau} \qquad
    \begin{ytableau}
        2 & 1 & 3 \\
        1 & 2
    \end{ytableau} \qquad
    \begin{ytableau}
        2 & 3 & 1 \\
        1 & 2
    \end{ytableau} \qquad
    \begin{ytableau}
        3 & 1 & 2 \\
        1 & 2
    \end{ytableau}\]
\end{example}

The wide partition conjecture for free matroids (Conjecture \ref{conj:wpcfm}) states that a partition $\lambda$ is Latin if and only if it is wide. Various results towards the WPC on free matroids are presented in \cite{chow2003wide}. Notably, they present an argument credited to Vic Reiner that shows any Latin partition $\lambda$ must be wide. So, it remains to show that if $\lambda$ is wide, then $\lambda$ is Latin. The authors reduce this problem by proving that:
\begin{enumerate}
    \item If all indecomposable wide partitions are Latin, then all wide partitions must be Latin.
    \item If all self-conjugate wide partitions are Latin, then all wide partitions are Latin.
\end{enumerate} Additionally, they prove the special cases:
\begin{enumerate}
    \item  If $\lambda$ is a wide partition with at most two distinct part sizes, then $\lambda$ is Latin.
    \item  If $\lambda$ is a wide partition with three distinct part sizes and either the second or third part size occurs with multiplicity one, or if $\lambda$ is a wide partition with four distinct part sizes and the second and fourth part sizes both occur with multiplicity one, then $\lambda$ is Latin.
\item If $\lambda$ is a self-conjugate wide partition with at most three distinct part sizes, then $\lambda$ is Latin.
\end{enumerate}
In \cite{allsop2025outlinerectanglesallocationslatin}, Allsop, Kotlar, and Wanless prove the stronger case:
\begin{itemize}
    \item[(4)] If $\lambda$ is a wide partition with three distinct part sizes, then $\lambda$ is Latin.
\end{itemize}

Latin tableaux are also studied in \cite{aharoni20252coverswideyoungdiagrams, chow2024Latintableauconjecture, KARPMAN2021102204}. Karpman and Roldán study isotopy graphs of Latin tableaux, which encode orbits of tableaux under a permutation group action. Chow and Tiefenbruck study the question of when Latin tableaux of shape $\lambda$ and type $\mu$ exist, which they propose an answer to with the Latin tableau conjecture. Aharoni, Berger, Guo, and Kotlar reformulate the WPC for free matroids in terms of hypergraph matching, and prove a weaker version in terms of hypergraph covers.

\subsection{Results on wide partitions}

We conclude this section by proving some additional results on wide partitions. 

\begin{proposition}
    If $\lambda$ is a wide partition of $n$, then the minimum value of $\lambda_1$ is given by the integer $j+1$ where $j^2 < n \leq (j+1)^2$. 
\end{proposition}

\begin{proof}
  Suppose $n \in \N$ is such that $j^2 < n \leq (j+1)^2$ for $j \in \Z_{\geq0}$. If $\lambda \vdash n$ is wide, then it must be contained inside the rectangle of size $\lambda_1' \times \lambda_1$. This implies $n \leq\lambda_1\lambda_1'$ and wideness forces $\lambda_1 \geq \lambda_1'$. Since $j^2 < n$, it must be that $j<\lambda_1$ and so $\lambda_1 \geq j+1$. 
  
  It remains to show that a wide partition $\lambda \vdash n$ with $\lambda_1=j+1$ exists. If $n=q(j+1)+r$ with $0\leq r\leq j$, then let $\lambda=((j+1)^q,r)$. The constraints on $n$ imply that $1 \leq q \leq j+1$ where it is only possible for $q=j+1$ if $r=0$ and $n=(j+1)^2$. We have constructed a partition $\lambda$ with $\lambda_1=(j+1)$ that we now show is wide. We check the lower subpartitions of $\lambda$. The partition $(r)$ is trivially wide, and the remainder take the form $\mu=((j+1)^t,r)$ for $1\leq t \leq q$.

  If $r=0$, then $\mu=(j+1)^t$ where $t\leq j+1$ and so $\mu \succeq \mu'$. Suppose $r \geq 1$, meaning $t \leq q \leq j$. In this case, $\mu' = ((t+1)^r, t^{j+1-r})$. Then, for $i \leq t$, we have $\mu_1 + \cdots + \mu_i = i(j+1) \geq i(t+1) \geq \mu_1' + \cdots + \mu_i'$. Finally, for $i \geq t+1$, we have $\mu_1 + \cdots + \mu_{i} = |\mu| \geq \mu_1' + \cdots + \mu_{i}'$.
\end{proof}

\begin{proposition} If $\lambda$ is a wide partition of $n$, then the maximum number of rows $\ell(\lambda)$ is given by the integer $k$ where $\frac{k(k+1)}{2} \leq n < \frac{(k+1)(k+2)}{2}$.
\end{proposition}

\begin{proof} 

The maximum number of rows for a wide partition of $n$ is equivalent to the number of rows in the biggest staircase shape that can be contained in a partition of $n$. This follows from the fact that $\lambda_{\ell(\lambda)-i}>i$ (staircase containment) \cite{chow2003wide}.  The biggest staircase of size $k$ that can fit within a partition of size $n$ is the one given by the $k$ such that $\frac{k(k+1)}{2} \leq n < \frac{(k+1)(k+2)}{2}.$
\end{proof}

It is straightforward to see that if $\lambda$ is a partition made up of distinct parts (the parts strictly decrease), then $\lambda$ is wide. Consider the diagram of a strict partition $\lambda$.  Filling each cell $(i,j)\in \lambda$ with $\lambda_i-j+1$ yields a Latin tableau. Since $\lambda$ is Latin, it must be wide. We can also show that a partition whose first row repeats but is otherwise strict is wide, which serves in the proof of our final result of this section.

\begin{lemma}\label{lem:strict1_dec}
    If $\lambda$ is a partition with $\lambda_1 = \lambda_2>\lambda_3 >\cdots > \lambda_{\ell(\lambda)}$ and $\lambda_1 \geq \lambda_1
'$, then $\lambda$ is wide and Latin.
\end{lemma}

\begin{proof}
Fill rows $2$ through $\ell(\lambda)$ as strict partitions are filled in the algorithm above (cell $(i,j)$ is filled with $\lambda_i-j+1$). Observe that rows $3$ through $\ell(\lambda)$ use $\ell(\lambda)-2$ of the $\lambda_1-1$ possible lengths less than $\lambda_1$. Since $\lambda_1 = \lambda_2 \geq \ell(\lambda)$, there must be some $1 \leq k<\lambda_1$ such that no row of $\lambda$ has length $k$.  Fill the first row with increasing entries from right to left starting with a 1 in column $k$. After filling the first column, loop back around to the final column and continue filling right to left.
This creates a Latin tableau of shape $\lambda$ and type $\lambda'$, so $\lambda$ is Latin and thus wide.
\end{proof}

In studying wide partitions, it is evident that many properties depend on the relative length of $\lambda_1$ and $\lambda_1'$.
    We say that a partition $\lambda = (\lambda_1, \ldots, \lambda_{\ell})$ has a \emph{square corner} if $\lambda_1 = \lambda_1'$. All self-conjugate partitions have square corners.

\begin{proposition}\label{prop:3.8}
    If $\lambda$ is a wide partition with a square corner, then $\lambda$ is indecomposable.
\end{proposition}

\begin{proof} Let $\lambda$ be a wide partition with a square corner and assume for contradiction that $\lambda = \mu + \nu$ for some wide partitions $\mu$ and $\nu$. Since $\lambda_1 = \ell(\lambda)$, we must have $\mu_1+\nu_1 = \ell(\lambda)$, meaning $0<\mu_1,\nu_1<\ell(\lambda)$. Since $\lambda_{\ell} > 0 $ either $\mu_{\ell} > 0$ or $\nu_{\ell} > 0$. In other words, either $\ell(\mu) = \ell(\lambda)$ or $\ell(\nu) = \ell(\lambda)$.  Without loss of generality,  assume $\ell(\mu) = \ell(\lambda)$. Then $(\mu')_1 = \ell(\lambda)$ and $\mu_1 < \ell(\lambda)$ which implies $\mu_1 < \mu'_1$. Thus $\mu \ngeq \mu'$, so $\mu$ is not wide, which is a contradiction. Therefore, we have shown that $\lambda$ is indecomposable.
\end{proof}

\begin{proposition}\label{prop:3.9} Let $\lambda$ be a wide partition with $\lambda_1 > \lambda_{1}'$. If $\lambda$ indecomposable, then there exists some $i \in [\ell(\lambda)-1]$ such that $\lambda_i = \lambda_{i+1}$.
\end{proposition}

\begin{proof} Let $\lambda$ be a wide partition with $\lambda_1 > \lambda_{1}'$. We proceed by proving the contrapositive. Suppose that $\lambda_1 > \lambda_2 > \lambda_3 > \cdots > \lambda_{\ell(\lambda)}$.  If $\lambda_1 - 1 > \lambda_2$ then $(\lambda_1 -1, \lambda_2, \ldots, \lambda_{\ell(\lambda)})$ is strict and, thus, wide.  If $\lambda_1 - 1 = \lambda_2$, then $(\lambda_1 -1, \lambda_2, \ldots, \lambda_{\ell(\lambda)})$ is wide by Lemma \ref{lem:strict1_dec}, since $\lambda_1 > \ell(\lambda)$ implies $\lambda_2 \geq \ell(\lambda)$.
Thus $\lambda$ is decomposable into wide partitions as $(\lambda_1 - 1, \lambda_2, \lambda_3, \ldots, \lambda_{\ell(\lambda)}) + (1).$
\end{proof}

Note that these Propositions \ref{prop:3.8} and \ref{prop:3.9} together imply that a strict wide partition $\lambda$ is indecomposable if and only if it is a staircase shape.

 \section{Latin fillings on families of wide partitions}\label{sec:WPC}

We prove Latinness constructively for three families of wide partitions. Our first theorem considers a large, natural class of wide partitions: those whose parts are all greater than or equal to their number of parts (these are referred to as \emph{squarish} in \cite{chow2003wide}). Our first result uses an existence proof via edge-list-coloring, while the remaining results in this section use explicit constructions.

\begin{theorem}\label{thm:fat}
    Let $\lambda = (\lambda_1, \ldots, \lambda_{\ell})$ be a partition. If $\lambda_{\ell} \ge \ell$, then $\lambda$ is Latin. 
\end{theorem}

\begin{proof}
Let $G_{\lambda}$ be a bipartite graph with vertices $V(G) = \{r_1, r_2, \ldots, r_{\ell}\} \cup \{c_1, c_2, \ldots, c_{\lambda_1}\}$.  For each cell $(i,j)\in \lambda$, let $G_{\lambda}$ have an edge $e_{ij}$ between $r_i$ and $c_j$. Note that $deg(r_i)=\lambda_i$ and $deg(c_j)=\lambda'_j$. A Latin tableau of shape $\lambda$ and type $\lambda'$ corresponds exactly to a proper edge-coloring of $G_{\lambda}$ in which the possible colors for $e_{ij}$ are given in the list $L(e_{i,j})=\{1,2, \ldots, \lambda_i\}$.

Given that $\lambda_{\ell} \geq\ell$, we have $\lambda_i \geq\lambda_{\ell} \geq \ell = \lambda_1' \geq \lambda'_j$ for all $i,j$.  Thus, for every $(i,j)\in \lambda$, \[|L(e_{ij})| = \lambda_i = max(\lambda_i, \lambda'_j) = max(deg( r_i),deg(c_j)).\] By Borodin-Kostochka-Woodall \cite[Thm 3]{boro}, $G_{\lambda}$ is ``edge-$f$-choosable" with $f(e_{ij})=\lambda_i$ and so admits a list-edge-coloring as desired. To obtain a Latin tableau from such a coloring, fill cell $(i,j)\in \lambda$ with the color used for edge $e_{ij}$. 
\end{proof}

Our next theorem considers a large subset of wide partitions that have exactly one repeated part. Lemma \ref{lem:strict1_dec} addressed a special case of this more general statement; specifically, the case where the two rows of the same length are the first and second rows. 

\begin{theorem}\label{thm:wide-strict-rep}
    Let $\lambda$ be a partition of length $\ell$ with $\ell-1$ distinct parts, and let $i$ be the part of $\lambda$ that appears twice. Let $M = \{m < i : m \not= \lambda_j, j\in[\ell]\}$. If there exists an $m \in M$ such that $m+i\not=\lambda_j$ for all $j\in[\ell]$, then $\lambda$ is Latin and wide.
\end{theorem}

\begin{proof}
    Let $\lambda$, $i$, and $m$ be as stated, and let $\lambda_{j}=\lambda_{j+1}=i$. We construct a Latin tableau of shape $\lambda$ as follows.
    For $x \not= j$, set $T(x,y)=\lambda_x-y+1$.  These rows all take a decreasing form: $(\lambda_x, \lambda_x-1, \ldots, 2,1 )$.
    For $y\leq m$, let $T(j,y)=m-y+1$. For $y > m$, let $T(j,y)= i - y + m +1$.
    Thus, the row $j$ has the form $(m, m-1, \ldots, 1, i, i-1, \ldots, m+1)$.

    Given that the rows with the exception of row $j$ have distinct lengths, there will be no repeated elements in any column among those rows. This is because two entries in the same column $y$ but different rows $a$ and $b$ will have values $\lambda_a-y+1$ and $\lambda_
    b-y+1$, which can't be equal since $\lambda_a \not= \lambda_b$. It remains to check that row $j$ does not create a conflict with any of these rows.  The entries in row $j$ all coincide with those that would occur in the decreasing row of length $m$ or length $m+i$. By assumption, these rows would have a distinct length from any others in the tableaux, and thus none of their entries would cause conflict.
    This filling produces a Latin tableau, and so $\lambda$ is Latin and thus also wide.
\end{proof}

Note that if $\lambda$ is wide, then the set $M$ is nonempty. If it weren't, $\lambda$ would have the subpartition $(i,i,i-1, \ldots, 2,1)$, which wouldn't dominate its conjugate as its first part is smaller than its length.

\begin{example}
 The algorithm in the proof above yields the following Latin tableau on $(7,5,5,4,2)$ where we pick $m=3$:
\[
 \begin{ytableau}
    7 & 6 & 5 & 4 & 3 & 2 & 1\\
    \mathbf{3} & \mathbf{2} & \mathbf{1} & \mathbf{5} & \mathbf{4}\\
   5 & 4 & 3 & 2 & 1\\
    4& 3 & 2 & 1\\
     2&1
 \end{ytableau}\] 
\end{example}

Next we consider shapes in which we can create a Latin filling that breaks $\lambda$ up into Latin rectangles.
\begin{theorem}\label{thm:gcd-mult}
    Let $\lambda = (\lambda_1, \ldots, \lambda_{\ell})$ be a partition with $w=\gcd(\lambda_1,\dots,\lambda_\ell)$, and let $m$ be the largest multiplicity of a part. If $w\ge m$, then $\lambda$ is Latin and wide.
\end{theorem}

\begin{proof} Since $w | \lambda_i$, each set of rows of the same length can be split into $\lambda_i/w$ consecutive blocks of width $w$ and height $\leq w$. 
To create a Latin tableau $T$ of shape $\lambda$, set $T(i,j) = (\lambda_i-wb) + 1 +((t_i + s) \mod w)$ where $b = \lceil j/w \rceil $ (its block), $s = (j-1) \mod w$ (its offset), and $t_i$ is the number of rows below row $i$ of the same length. 

First, we verify there is no repetition within any row of $T$. Consider some row $i$.  In block $b$, the values of $s$ are given by $\{0, \ldots, w-1\}$ so the values of $((t_i +s) \mod w)$ are also $\{0, \ldots, w-1\}$ and the values of $1+((t_i +s) \mod w)$ are $\{1,2,\ldots, w\}$. The entries in block $b$ then span values $\{ \lambda_i -wb+1, \ldots, \lambda_i-wb+w\}$. The possible values of $b$ are $\{1,2,\ldots, \lambda_i/w\}$, so there is no overlap between the entries in the different blocks. 

Next, we check that there are no repeats within columns. 
Consider row $i$ and row $h$, where $i <h$. Case 1: Let $\lambda_i=\lambda_h$. For entries in some column $j$ in row $i$ and $h$, the values of $b$ and $s$ are the same while $t_i \not=t_h$ and $t_i, t_h \in \{0, \ldots, m-1\}$ where $m - 1 \leq w-1$.  Thus, $(t_i + s) \mod w$ is not equal to $(t_h + s) \mod w$. It follows that $T(i,j) \not= T(h,j)$. 

Case 2: Let $\lambda_i > \lambda_h$.  Then, $T(i,j)-T(h,j)=(\lambda_i-\lambda_h)+ ((t_i +s)) \mod w - ((t_h)+s) \mod w$. Observe that $((t_i +s)) \mod w - ((t_h)+s) \mod w < w$ and $\lambda_i - \lambda_h \geq w$.  Therefore, $T(i,j)-T(h,j)>0$. Thus, we have shown for both cases that there cannot exist two entries of the same value in any column of $T$.
So, $T$ is a Latin tableau of shape $\lambda$. Having shown that $\lambda$ is Latin, it follows that $\lambda$ is wide. 
\end{proof}

\begin{example}
    Let $\lambda = (6,4,4,2,2)$ with $w=2$ and $m=2$.  The proof above yields the following Latin filling of $\lambda$:
\[ \begin{ytableau}
    5 & 6 & 3 & 4 & 1 & 2 \\
    4 & 3 & 2 & 1 \\
    3 & 4 & 1 & 2 \\
    2 & 1\\
    1 & 2 
\end{ytableau} \]
\end{example}

Given any partition, including one that is not wide, we can easily create a partition to which Theorem \ref{thm:gcd-mult} will apply.

\begin{corollary}
    If $\mu = (\mu_1, \ldots, \mu_{\ell})$ is any partition and $w$ is a positive integer greater than or equal to the largest multiplicity of any part in $\mu$, then $w\mu=(w\mu_1, \ldots, w\mu_{\ell})$ is Latin and wide.
\end{corollary}

We close this section with a result on the existence of \emph{symmetric Latin tableaux}, which we define to be Latin tableaux invariant under conjugation. Note that by this definition, only self-conjugate partitions can possess symmetric fillings.

\begin{proposition}\label{prop:symfill}
    If a self-conjugate partition $\lambda$ admits a symmetric Latin filling of type $\lambda'$, then the number of odd parts of $\lambda$ is at most its Durfee size. 
\end{proposition}

\begin{proof}
Let $T$ be a symmetric Latin tableau of shape $\lambda$ and type $\lambda'$. By definition, $T(i,j)=T(j,i)$ for all $i \not= j$. Thus, every element $k$ occurs off the main diagonal an even number of times in $T$. Specifically, $k$ occurs a total of $\lambda'_k = \lambda_k$ times. So, each odd part of $\lambda$ corresponds to an element that must occur on the diagonal,  whose length is given by the Durfee size. Thus, the number of odd parts of $\lambda$ must be less than or equal to the Durfee size.
\end{proof}

The condition in Proposition \ref{prop:symfill} is not sufficient for the existence of a symmetric Latin filling. For example, consider $\lambda=(7,7,5,5,4,2,2)$, which has Durfee size $4$ and exactly $4$ odd parts. There will be an odd number of $1$s, $2$s, $3$s, and $4$s, meaning each of those numbers must appear on the diagonal. It is then impossible to place the $6$s and $7$s in the top two rows in a symmetric way.

\section{Enumeration of Latin Tableaux}\label{sec:LT}

An enumerative formula for Latin tableaux of shape $\lambda$ and type $\lambda'$ provides a way to test the Latinness of a partition without finding a specific Latin filling, as $\lambda$ is Latin if and only if $LT(\lambda)>0$. Enumeration of Latin tableaux is also an interesting problem in its own right, and a natural extension of the problem of enumerating Latin squares.
In our main theorem of this section, we present a formula for $LT(\lambda)$, the number of Latin tableaux of shape $\lambda$ and type $\lambda'$, using $(0,1)$-matrices and matrix permanents. In the case where $\lambda$ is a square, we recover a classical result on Latin squares by Shao-Wei \cite{SHAO1992293}:
\[L_n = \sum_{A \in B_n} (-1)^{\sigma_0(A)} (per(A))^n,\] where $B_n$ is the set of $n\times n$ (0,1)-matrices, $\sigma_0(A)$ is the number of zero entries in matrix $A$, and $per(A)$ is the permanent of matrix $A$.

 Let $S_{k \times n}(t)$ be the set of $(0,1)$-matrices of size $k \times n$ where each row and column has at most one non-zero entry, and the first $t$ rows always contain a 1 while the remaining rows are zero. Let $J_{\lambda}$ be the $\ell(\lambda) \times \lambda_1$ matrix where the first $\lambda_i$ entries in row $i$ are 1 and all other entries are 0. Let $\rr_i(\lambda) = \max\{j : \lambda_j \geq i\}$, so $\rr_i(\lambda)$ is the number of rows in $\lambda$ with length of at least $i$. 

\begin{lemma} Let $\lambda$ be a partition with $\lambda_1 \geq \ell(\lambda)$.  Latin tableaux of shape $\lambda$ and type $\lambda'$ are in bijection with sequences of $\lambda_1$ matrices $(M_1, M_2, \ldots, M_{\lambda_1}) \in S_{\ell(\lambda) \times \lambda_1}(\rr_1(\lambda)) \times S_{\ell(\lambda) \times\lambda_1}(\rr_2(\lambda)) \times \cdots \times S_{\ell(\lambda) \times\lambda_1}(\rr_{\lambda_1}(\lambda))$ such that $M_1 + M_2 + \cdots + M_{\lambda_1} = J_{\lambda}$.
\end{lemma}

\begin{proof}
    Let $T$ be a Latin tableau of shape $\lambda = (\lambda_1, \ldots, \lambda_{\ell})$. The entries in $T$ are given exactly by the set $\{1, 2, \ldots, \lambda_1\}$. The rows that will contain an entry $i$ are exactly rows $1$ through $\rr_i(\lambda)$.  Thus, we can encode the positions occupied of $i$ in $T$ by a matrix $M_{i} \in S_{\ell(\lambda) \times\lambda_1}(\rr_i(\lambda))$ that is entry-wise less than or equal to $J_{\lambda}$. Since each position in $\lambda$ must be occupied, the sum of these matrices will yield $J_{\lambda}$.  
\end{proof}

Let $\mathfrak{S}_m^k$ be the set of $k$-permutations of $[m]$.
Define the $k$-row permanent of an $\ell \times m$ matrix $A$ to be 

\[ per_k(A) = \sum_{i_1, \ldots, i_k\in \mathfrak{S}_m^k} a_{1,i_1}a_{2,i_2} \cdots a_{k,i_k}   \]

This is equivalent to taking the permanent but only considering the first $k$ rows of the matrix.

\begin{theorem}\label{thm:main} For a partition $\lambda = (\lambda_1, \ldots, \lambda_{\ell})$ with $\lambda_1 \geq \ell$, let $B(\lambda)$ be the set of $\ell \times \lambda_1$ $(0,1)$-matrices entrywise less than or equal to $J(\lambda)$, and let $\sigma^{\lambda}_0(A)$ be the number of zero elements of the matrix $A$ whose indices occur in $\lambda$. Then, the number of Latin tableaux of shape $\lambda$ is given by 
    \[LT(\lambda) = \sum_{A \in B(\lambda)} (-1)^{\sigma^{\lambda}_0(A)}\prod_{i \in [\lambda_1]}per_{\rr_i(\lambda)}(A)\]
\end{theorem}

\begin{proof} Consider the sets \[S(\lambda) = \{(M_1 , M_2, \ldots, M_{\lambda_1}) \in S_{\ell \times \lambda_1}(\rr_1(\lambda)) \times S_{\ell \times\lambda_1}(\rr_2(\lambda)) \times \cdots \times S_{\ell \times\lambda_1}(\rr_{\lambda_1}(\lambda)) : M_i \leq J_{\lambda}, i \in [\lambda_1] \},\] 
 \[\text{ and } S(\lambda)_{ij} = \{(M_1, M_2, \ldots, M_{\lambda_1})\in S(\lambda) : (M_1)_{ij} = \cdots = (M_{\lambda_1})_{ij} = 0\} \text{ for } (i,j)\in \lambda,\]  while $\overline{S(\lambda)_{ij}} = S(\lambda) \setminus S(\lambda)_{ij}$.
    By our bijection above, a Latin tableau of shape $\lambda = (\lambda_1, \ldots, \lambda_{\ell})$ corresponds to a sequence of matrices $(M_1, M_2, \ldots, M_{\lambda_1}) \in S(\lambda)$ satisying $M_1 + \cdots M_{\lambda_1} = J(\lambda)$. 
    Note that $M_1 + \cdots M_{\lambda_1} = J(\lambda)$ if and only if $(M_1)_{ij}, \cdots, (M_{\lambda_1})_{ij}$ are not all zero for each $(i,j)\in \lambda$. This is because the total sum of all entries in $M_1, \ldots, M_{\lambda_1}$ will be exactly $|\lambda|$ and our entries are all 1, so the only possible way for at least one matrix to have a non-zero entry for each $(i,j)\in \lambda$ is if exactly one matrix has a 1 in position $(i,j)$ for each $(i,j)\in \lambda$.
    Next, $(M_1)_{ij}, \cdots, (M_{\lambda_1})_{ij}$ are not all zero for each $(i,j)\in \lambda$ is true if and only if $(M_1, \ldots, M_{\lambda_1})\in \cap_{(i,j)\in \lambda} \overline{S(\lambda)_{ij}}$ by definition.  
    Therefore, \[LT(\lambda) = |\bigcap_{(i,j)\in \lambda} \overline{S(\lambda)_{ij}}|.\]
    Given a subset $X \subseteq \lambda$, let $S_X(\lambda) = \cap_{(i,j)\in X} S(\lambda)_{ij}$.
    By the principle of inclusion-exclusion, 
    \[LT(\lambda) = \sum_{X \subseteq \lambda} (-1)^{|X|}|S_X(\lambda)|.\] It remains to find a fomula for $|S_X(\lambda)|$. 
    
    Let $A(X)$ be the $(0,1)$-matrix where $A(X)_{ij}=0$ if $(i,j)\in X$ or $(i,j)\not\in \lambda$ and $A(X)_{ij}=1$ otherwise.  The set $S_X(\lambda)$ is exactly the set of $\lambda_1$-tuples of matrices where each matrix is less than $J(\lambda)$ and  the $i^{\text{th}}$ entry $M_i \in S_{\ell \times \lambda_1}(\rr_i(\lambda))$ has 0s at each $(i,j)\in X$. Further, since $M_i \in S_{\ell \times\lambda_1}(\rr_i(\lambda))$, it will be a matrix with a single 1 in rows $1$ through $\rr_i$ and only zeroes in all other rows, with at most one 1 per column. So, these $M_i$ correspond to ways of picking the $\rr_i$ entries that will be equal to 1 while making sure to avoid indices in $X$.
    Thus, the number of possible $M_i$ matrices for an element $(M_1, \ldots, M_{\lambda_1}) \in S_X(\lambda)$ is given by $per_{\rr_i(\lambda)}(A(X))$ for each $i \in [\lambda_1]$.  Therefore, \[|S_X(\lambda)| = \prod_{i \in [\lambda_1]} per_{\rr_i(\lambda)}(A(X)).\]
    This yields the formula \[LT(\lambda) = \sum_{X \subseteq \lambda} (-1)^{|X|}\prod_{i \in [\lambda_1]} per_{\rr_i(\lambda)}(A(X))\] which simplifies to the stated result. 
\end{proof}

Let $LT(n) = \sum_{n \vdash \lambda} LT(\lambda)$. The values of $LT(n)$ for $n \geq 1$ are:
\[1,2,7,30,144,857,5874,46342,410216,\dots\] 

Our formula can also be applied to rectangle and square shapes.

\begin{corollary}
    For $\lambda = (n)^k$ with $n \geq k$, our formula specializes to Formula (14) for the number of Latin rectangles in Remark 2 of \cite{SHAO1992293}:
    \[LT((n)^k) = \sum_{A \in B_{k \times n} }(-1)^{\sigma_0(A)}(per(A))^n.\]
    When $n=k$, this becomes Formula (15) for the number of Latin squares in Remark 2 of \cite{SHAO1992293}.
\end{corollary}

We also separately consider enumeration of certain special cases. 
\begin{proposition} For $n,k \in \mathbb{N}$ with $0 < k \leq n$, \[LT((n,k)) = k! \sum_{i=0}^k (-1)^i \binom{k}{i}(n-i)!\]
    
\end{proposition}

\begin{proof} 

Consider a filling of the second row of $\lambda$, which can be associated with a permutation of length $k$. Given this permutation $\sigma$, we count all permutations $\pi$ of length $n$ that can be used to fill the top row of $\lambda$ and create a Latin tableau. These are exactly the permutations where $\pi_i \not = \sigma_i$ for all $i\in[k]$. We will denote this set $L_k$.
 
 Let $\sigma$ be the identity permutation $\sigma = (1,2, \ldots, k)$.
Let $N_i$ be the set of permutations $\pi$ of length $n$ where $\pi_i=i$. It is clear that if $\pi \in N_i$ for any $i\in [k]$, then $\pi \not \in L_k$. Moreover, any $\pi \not\in L_k$ must be in $N_i$ for some $N_i$. So, $L_k = S_n \setminus (N_1 \cup N_2 \cup \cdots \cup N_k)$.

By the principle of inclusion-exclusion, we have 
\[   |N_1 \cup N_2 \cup \cdots \cup N_k| = \sum_{i_1 \in [k]} |N_{i_1}| - \sum_{i_1 > i_2 \in [k]} |N_{i_1} \cap N_{i_2}| +  \cdots  + (-1)^k \sum_{i_1 > \cdots > i_k \in [k] }|N_{i_1} \cap N_{i_2} \cap \cdots \cap N_{i_k}|.\]
  Observe that $|N_{i_1} \cap N_{i_2} \cap \cdots \cap N_{i_j}| = (n-j)!$ and there are $\binom{k}{j}$ choices of $i_1, i_2, \ldots, i_j$. Therefore, \begin{align*}
    |N_1\cup\cdots \cup N_k| = \binom{k}{1}(n-1)! - \binom{k}{2}(n-2)! + \binom{k}{3}(n-3)! + \cdots + \binom{k}{k}(n-k)!
\end{align*} 

To find $|L_k|$ we do $|S_n| - |N_1 \cup N_2 \cup \ldots, \cup N_k|$. Noting that $n! = \binom{k}{0}n!$, this simplifies to \begin{align*}
    n! - |N_1 \cup \cdots \cup N_k| & = n! - \Big(\binom{k}{1}(n-1)! - \binom{k}{2}(n-2)! + \cdots + \binom{k}{k}(n-k)!\Big)
    \\ & = \binom{k}{0}(n-0)! - \binom{k}{1}(n-1)! + \binom{k}{2}(n-2)!  + \cdots (-1)^k\binom{k}{k}(n-k)!
    \\ & = \sum_{i=0}^k (-1)^i\binom{k}{i}(n-i)!
\end{align*}

This counts the number of Latin tableaux of shape $(n,k)$ where the second row is filled with $\sigma = (1,2,\ldots, k)$. This process is independent of $\sigma$, so we multiply by a factor of $k!$ to account for all permutations. 
\end{proof}

Next, we consider staircase shapes and staircase shapes with the top row extended.

\begin{proposition}
    There exists a single Latin tableau of shape $\lambda$ and type $\lambda'$ if and only if $\lambda$ is the staircase shape $(k, k-1, \ldots,2,1)$ for some $k \in \N$.
\end{proposition}

\begin{proof}
    Consider a staircase shape $\lambda = (n-1,n-2, \ldots, 2,1)$. To create a Latin tableau of type $\lambda'$, one is forced to fill the bottom row with a 1. This forces the next row to be filled with a 2 and then a 1 (left to right), which in turn forces the next row to be filled with 3,2, and 1. Continuing in this way gives a Latin filling of $\lambda$ which must be unique.

    For the other direction, first observe that $\lambda$ is a staircase if and only if both $\lambda$ and $\lambda'$ are strict. Thus, if $\lambda$ is not a staircase, then there is either a repeated row or a repeated column. If any Latin filling of type $\lambda'$ exists, there will be at least one more filling obtained by switching those rows or columns of the same length. 
\end{proof}

\begin{proposition}
    If $\lambda$ is a partition of shape $(k+t,k-1,k-2, \ldots, 3,2,1)$ with $t \geq 0$, then \[LT(\lambda) = t!(t+1)^k.\]
\end{proposition}

\begin{proof}
    The bottom $k-1$ rows of $ \lambda$ form a staircase of size $k-1$ and thus box $(i,j)$ for $i>1$ must be filled with $\lambda_i-j+1$. It remains to count the number of ways to fill the first row. The first entry of row 1 could be any element of $\{k,k+1, \ldots, k+t\}$. The second entry of row 1 could be any element of $\{k-1,k,k+1, \ldots, k+t\}$ except for whichever element fills the first entry of row 1. In this way, there are $(t+1)$ possible ways to fill each of the first $k$ cells of row $1$, giving a total of $(t+1)^k$ fillings.  The remaining $t$ cells must be filled with the remaining $t$ elements, which can be ordered in $t!$ ways. 
\end{proof}

We may also understand the number of all Latin tableaux in terms of the number of \emph{reduced Latin tableaux}.
We introduce two different notions of reduced Latin tableaux to generalize reduced Latin squares \cite{cayley1890Latin, macmahon1915combinatory}. 

\begin{definition}
    A Latin tableau is \emph{column-reduced} if the top-most entries among each set of columns of the same length increase from left to right. A Latin tableau is \emph{row-reduced} if the leftmost entries among each set of rows of the same length increase from top to bottom.
\end{definition}

Let $RLT^{col}(\lambda)$ and $RLT^{row}(\lambda)$  denote the number of column-reduced and row-reduced Latin tableaux of shape $\lambda$, respectively. 

\begin{proposition}\label{prop:divis} Let $\lambda = (i_1^{a_1}, i_2^{a_2}, \ldots, i_k^{a_g} )$  be a partition with $\lambda' = (j_1^{b_1}, \ldots, j_h^{b_h})$. Then,

\[LT(\lambda) =b_1!b_2! \cdots b_h! RLT^{col}(\lambda)= a_1!a_2! \cdots a_g! RLT^{row}(\lambda).\]
\end{proposition}

\begin{proof} 
Consider a column-reduced Latin tableau $R$ of shape $\lambda$ with $\lambda' = (j_1^{b_1}, \ldots, j_h^{b_h})$. There are $b_1!b_2!\cdots b_h!$ permutations of the columns of $R$ to obtain a Latin tableau of shape $\lambda$. No two columns are the same, so each of these permutations is unique. Additionally, every Latin tableau is a permutation of a unique column-reduced Latin tableau. This yields the first part of our result. Applying similar logic to the row-reduced tableaux yields the second part. 
\end{proof}

\section{Comparing values of $LT(\lambda)$}\label{sec:mono}

 In cases where we can show that $LT(\lambda) \geq LT(\mu)$, the problem of proving that $\lambda$ is Latin reduces to one of proving that $\mu$ is Latin. We take this approach for certain types of wide partitions here. 
Specifically, we show $LT(\lambda)\geq LT(\mu)$ by proving that each Latin filling of $\mu$ corresponds to at least one Latin filling of $\lambda$.  

\begin{proposition}
    If $\lambda$ and $\mu$ are wide partitions where $(\lambda, 1) = \mu$, then $LT(\lambda) \geq LT(\mu)$.
\end{proposition} 

\begin{proof}
Each Latin tableau of shape $\mu$ corresponds to a Latin tableau of $\lambda$ that does not contain 1 in the first column. This is injective because the bottom single-box row of any Latin tableau of shape $\mu$ must contain a 1, making the rest of the tableau one of shape $\lambda$ that does not have (another) 1 in the first column.
\end{proof}

\begin{proposition}
    Let $\lambda$ be a partition with $\lambda_1 \geq max\{\lambda_i+i-1 : 2 \leq i \leq \ell(\lambda)\}$, and let $\mu=(\lambda_2, \ldots, \lambda_{\ell})$. Every Latin tableau $T$ of shape $\mu$ is a lower subtableau of some Latin tableau of shape $\lambda$. 
\end{proposition}

\begin{proof}

Let $T$ be a Latin tableau of shape $\mu$. For each $i\in[\lambda_1]$, let $U_i$ be the set of values that appear in column $i$ of $T$ (note $U_i=\emptyset$ when $i>\lambda_2=\mu_1$).
Each possible row of length $\lambda_1$ that could be added to the top of $T$ corresponds to a permutation $\pi = \pi_1 \pi_2 \cdots \pi_{\lambda_1}$ where $\pi_i \not\in U_i$ for each $i \in [\lambda_1]$. Equivalently, this corresponds to a perfect matching in a bipartite graph $G$ with vertices $c_1, c_2, \ldots, c_{\lambda_1}$ corresponding to columns of the new row, vertices $1, 2, \ldots, \lambda_1$ corresponding to values that need to appear in the row, and edges $(c_i,j)$ if and only if $j\not\in U_i$. By Hall's Theorem \cite{Hall1935}, to prove a perfect matching exists, it suffices to show that for any set $I \subset [\lambda_1]$, we have $|\cup_{i \in I} N(c_i)| \geq |I|$, where $N(c_i)$ is the set of neighbors of $c_i$.

Assume for contradiction that there exists a set $I \subset [\lambda_1]$ such that $|\cup_{i \in I} N(c_i)| < |I|$. Note that $\cup_{i \in I} N(c_i) = \cup_{i \in I} ([\lambda_1] \setminus U_i)$. Let $M = [\lambda_1] \setminus \cup_{i \in I} N(c_i)$. Then, $|M| = \lambda_1 - |\cup_{i \in I} N(c_i)| > \lambda_1 - |I|$, so $|I|+|M| \geq \lambda_1 + 1$. 
By definition, every $v\in M$ has $v\in U_j$ for all $j \in I$ (meaning, $v$ occurs in each of the selected columns).
In $T$, a value $v$ occurs in row $i$ of $\mu$ exactly when $\mu_i \geq v$, or equivalently, $\lambda_{i+1} \geq v$. Thus, $v$ occurs exactly $|\{i\geq 2: \lambda_i \geq v\}| = \lambda'_v-1$, each in a distinct column.  Given that $v$ occurs in exactly $\lambda_v'-1$ columns, and it occurs in every column indexed by $I$, we must have $|I| \leq \lambda_v'-1$, for every $v\in M$.

Given that $|M| > \lambda_1 - |I|$, we have $\lambda_1 + 1 \leq |I| + |M|$. Let $w = max(M)$ and note $|M| \leq w$, then
\[\lambda_1+1 \leq |I|+|M| \leq \lambda_w'-1+w.\]
Let $r = \lambda_w'=|\{i : \lambda_i \geq w\}|$, meaning $\lambda_r \geq w$. Since $I$ is nonempty and $|I|\leq \lambda'_w-1$, we have $r \geq 2$. Then, by definition of $\lambda$, we have $\lambda_r \leq \lambda_1 - r +1$. Therefore, \[\lambda_w' + w \leq \lambda_r + r \leq \lambda_1+1 .\]
This contradicts our previous fact that $\lambda_1+1 \leq (\lambda_w'-1)+w$. So, Hall's condition holds, and a perfect matching exists, meaning there exists a row that, when added to $T$, creates a Latin tableau of shape $\lambda$.
\end{proof}

\begin{corollary}\label{cor:addrow}
If $\lambda = (m,\mu)$ where $m \geq max\{\mu_i+i :i \in [\ell(\mu)]\}$, then $LT(\lambda) \geq LT(\mu)$.
\end{corollary}

The partition $\mu=(5,4,3,3)$ is Latin by \cite[Thm 4]{chow2003wide} and has $\max(\mu_i+i)=7$. Then by Corollary \ref{cor:addrow}, $\lambda = (k, \mu)$ for any $k \geq 7$ is Latin. These partitions are not shown to be Latin as part of any other special family up to this point. 

Our second angle builds on the work in Section \ref{sec:LT} counting Latin tableaux with matrix permanents. We begin with a formula that leverages strategies for extending Latin rectangles from \cite{stones2010many}.
For a partition $\lambda$ and $r \in [\ell(\lambda)]$, let $\lambda^{(r)}$ be the partition obtained by deleting row $r$ of $\lambda$. Let $\mathcal{LT}(\lambda)$ denote the set of Latin tableaux of shape $\lambda$.
The permanent of an $n \times n$ matrix $A$ is defined as 
\[per(A) = \sum_{\sigma\in \mathfrak{S}_n} a_{1,\sigma(1)}a_{2,\sigma(2)} \cdots a_{n,\sigma(n)}.\]
Using permanents, we construct a recursive expression for the number of Latin tableaux of shape $\lambda$ in terms of Latin tableaux of $\lambda^{(r)}$.
\begin{lemma} For a partition $\lambda$ and any $r \in [\ell(\lambda)]$, 
\[LT(\lambda) = \sum_{T \in \mathcal{LT}(\lambda^{(r)})} per(A_{\lambda_r}(T)),\]
where $A_{\lambda_r}(T)$ is the availability matrix of $T$, the $(0,1)$-matrix of size $\lambda_r \times \lambda_r$ where $a_{ij} =1$ when $j$ does not occur in column $i$ of $T$.
\end{lemma}

\begin{proof} Every Latin tableau $T$ of shape $\lambda$ corresponds to a Latin tableau $T'$ of shape $\lambda^{(r)}$ with an additional row of size $\lambda_r$ that causes no conflicts inserted at position $r$. The permanent of the availability matrix counts exactly the number of such rows that may be inserted at position $r$ for each $T'$.
\end{proof}

\begin{proposition} If $\lambda$ and $\mu$ are wide partitions where $\lambda = (\mu_1+1, \mu_2, \ldots, \mu_{\ell})$, then $LT(\lambda) \geq LT(\mu)$. 
\end{proposition}

\begin{proof}
    Observe that $\lambda^{(1)} = \mu^{(1)}$ and so the set $\mathcal{LT}(\lambda^{(1)}) = \mathcal{LT}(\mu^{(1)})$. 
    Consider some Latin tableau $T\in \mathcal{LT}(\mu^{(1)})$ with the availability matrix $A_{\mu_1}(T)$ of size $\mu_1 \times \mu_1$.  The same Latin tableau $T\in \mathcal{LT}(\lambda^{(1)})$ will have an availability matrix $A_{\mu_1+1}(T)$ of size $\mu_1 + 1 \times \mu_1 + 1$ that matches $A_{\mu_1}(T)$ on all coordinates $(i,j)$ where $i,j \leq \mu_1$. Since the top row of size $\mu_1+1$ must be strictly longer, the entry $\mu_1+1$ does not occur in any column of $T$ and so the entry at $(i,\mu_1+1)$ is $1$ for all $i \leq \mu_1+1$. Similarly, the $\mu_1+1$ column of $T$ is empty so no value is contained in it. As such, the entry at $(\mu_1+1, j)$ is equal to $1$ for all $j \leq \mu_1+1$.  
    Each permutation $\sigma$ in the calculation of the permanent of $A_{\mu_1}(T)$ can be paired with the permutation $\sigma'=(\sigma_1, \ldots, \sigma_{\mu_1}, \mu_1+1)$ in the calculation of the permanent of $A_{\mu_1+1}(T)$ . Since the entry at $(\mu_1+1, \mu_1+1)=1$, if the term corresponding to $\sigma$ contributes 1 to the permanent of $A_{\mu_1}(T)$, the term corresponding to $\sigma'$ contributes 1 to the permanent of $A_{\mu_1+1}(T)$. It follows that $LT(\lambda) \geq LT(\mu)$. \end{proof}

Similar operations on partitions do not always yield an increase in the number of Latin Tableaux:
\begin{enumerate}
    \item adding a box to the bottom row (for example, $LT((4,4,2,2)) = 8 < 12 = LT((4,4,2,1))$),
    \item adding a top row of equal length (for example, $LT((3,3,2)) = 4 < 6 = LT((3,2))$), and
    \item adding a strictly longer top row (for example, $LT((5,4,4,3,1)) = 32 < 40 = LT((4,4,3,1))$).
\end{enumerate}

One may also suspect that the number of Latin tableaux for wide partitions always weakly increases moving up in the subposet formed by dominance order on wide partitions, but this holds for $n \leq 16$ and fails at $n=17$. At $n=17$, we see that $(6,5,3,2,1) \preceq (7,4,3,2,1)$, but $LT((6,5,3,2,1)) = 584 > 486 = LT((7,4,3,2,1))$.

\section{Open problems}\label{sec:open}
The following two conjectures have been verified computationally for partitions of size $n \leq 22$.

\begin{conjecture}\label{conj:box}
If $\lambda$ and $\mu$ are wide partitions where $\lambda_i = \mu_i + 1$ for some $i\in [\ell(\lambda)-1]$ and $\lambda_j=\mu_j$ when $j \not=i$, then $LT(\lambda) \geq LT(\mu)$.
\end{conjecture}

The partition $ \lambda = (6,5,5,3,3,1)$ is an example of a new partition that would be covered by Conjecture \ref{conj:box}. The partition $\mu=(6,5,5,3,2,1)$ is Latin by Theorem \ref{thm:wide-strict-rep}, so $\lambda$ would also be Latin, as it is obtained by adding a box to row $5$ of $\mu$. Every wide partition up through $n \leq 27$ would be covered, with the exception of $(6,5,4,3,3)$. The following similar conjecture is of interest and verified computationally through $n \leq 22$, though it is less effective for proving Latinness. 

\begin{conjecture}
If $\lambda = (1)^k + \mu$ with $\mu_1' \leq k \leq \mu_1'+1$, then $LT(\lambda)\geq LT(\mu)$.
\end{conjecture}

The smallest wide partition not covered by current results is $\lambda=(6,5,4,3,3)$. There are 576 Latin tableaux of shape $\lambda$ and type $\lambda'$, one being: \vspace{1mm}
\[
\begin{ytableau}
    6&5&4&3&2&1\\
    4&1&5&2&3\\
    2&4&3&1\\
    1&3&2\\
    3&2&1
\end{ytableau}\] \vspace{2mm}

Conjecture \ref{conj:box} would not prove that $\lambda$ is Latin as none of the partitions $(5,5,4,3,3)$, $(6,4,4,3,3)$, or $(6,5,3,3,3)$ are wide.  Thus, a new technique is needed. All fillings of $\lambda$ have $4$, $5$, and $6$ in the first three cells of the top row, while $1$, $2$, and $3$ are always in the last three cells of the top row. Nothing in the class of partitions of the form $(k,k-1, \ldots, 5,4,3,3)$ for $k \geq 6$ is known to be Latin by any result, but Latin tableaux on these shapes appear to be countable by a simple formula.

\begin{conjecture}
    If $\lambda = (k,k-1, \ldots, 5,4,3,3)$ with $k\geq 5$, then $LT(\lambda)=9\cdot4^{k-3}$.
\end{conjecture}

Similar shapes that are not currently proven to be Latin by any results include $(8,7,5,4,4,4)$, $(8,6,5,5,4,4)$, $(7,6,6,5,4,4)$, $(7,7,6,4,4,2,2)$, $(7,7,5,5,4,2,2)$, and $(7,7,6,5,3,3,2)$. Understanding these shapes appears to require different techniques, and so they constitute natural next test cases for the wide partition conjecture on free matroids.


\subsection*{AI Declaration} Claude Opus 5.5 was used to refine conjectures, generate and analyze data, review the paper, and, in some cases, assist with proofs. 
The authors take full responsibility for the correctness and integrity of all mathematical results, proofs, and text.

\printbibliography

\end{document}